\documentclass{article}
\usepackage{amssymb,amsmath,amsfonts,amsthm,latexsym}
\usepackage{multicol}
\usepackage{graphicx}
\usepackage{multicol}
\usepackage[english]{babel}
\usepackage[usenames]{color}
\newtheorem{theorem}{Theorem}

\newtheorem{theo}{Theorem}

\newtheorem{lemm}{Lemma}
\newtheorem{coro}[lemm]{Corollary}
\newtheorem{prop}[lemm]{Proposition}

\newtheorem{rema}[lemm]{Remark}
\newtheorem{defi}[lemm]{Definition}

\newtheorem*{lemm*}{Lemma}
\newtheorem*{clai*}{Claim}

   \def\PP{{\mathbb P}}
 \def\RR{{\mathbb R}}

   \def\cN{{\cal N}} 
\def\cC{{\cal C}}    \def\cU{{\cal U}}

    \def\cX{{\cal X}}

\def\set#1{\left\{\, #1 \,\right\}}

\def\norm #1{\Vert \,#1\, \Vert\,}

\title{ Star vector fields on three-manifolds are multi-singular hyperbolic.}

\author{Jennyffer Bohorquez,  Adriana da Luz, Nelda Jaque} 

\begin{document}
\maketitle

\begin{abstract}

The coexistence of singularities and regular orbits in chain transitive sets has been a major obstacle in understanding the hyperbolic/partial hyperbolic nature of robust dynamics. Notably, the vector fields with all periodic orbits robustly hyperbolic (star flows), are hyperbolic in absence of singularities.

Morales, Pacifico and Pujals proposed  a partial hyperbolicity called ``singular hyperbolicity" that characterizes an open and dense subset of three dimensional star flows. In higher dimensions, Bonatti and da Luz characterize an open and dense set of star vector fields by multi-singular hyperbolicity.
In this article, we prove that a vector field exhibiting all periodic orbits robustly of the same index is multi-singular hyperbolic, generalizing the previous results. As a corollary, we obtained that all three-dimensional star flows are multi-singular hyperbolic. Moreover, if all singularities in the same class exhibit the same index, the star flow is singular hyperbolic. Additionally, star flows with robust chain recurrence classes in any
dimension are  multisingular hyperbolic.
\end{abstract}



\section{Introduction}



Since the 60's there has been a considerable amount of interest and effort in understanding which vector fields or diffeomorphisms are stable  (i.e., systems whose topological dynamics are
unchanged under small perturbations). When not, studies investigated the closeness of a system  to stability.
A major question has been the characterization of the structural stability  by  hyperbolicity (i.e., a global structure  in the tangent space expressed by expansion and contraction of transverse subspaces).
This characterization, first stated in the stability conjecture \cite{PaSm}, was proven for diffeomorphisms in $C^1$ topology by Robin and Robinson in \cite{R1}, \cite{R2}
(hyperbolic systems are structurally stable) and Ma\~{n}\'e \cite{Ma2} (structurally stable systems are hyperbolic).

The dynamics of flows are believed to be  related to the dynamics of diffeomorphisms. However, a phenomenon exists that is specific to vector fields: the existence of singularities (zeros of the vector field) accumulated, robustly, by regular recurrent orbits. This setting is more challenging,
indicating why the equivalent result for flows in the stability conjecture (also for the $C^1$ topology) was proven in \cite{H2} nearly ten years after the results in \cite{Ma2}.

An example of this behavior has been indicated by Lorenz in \cite{Lo}. Guckenheimer and Williams, in \cite{GW}, constructed
a $C^1$-open set of vector fields in a three-manifold, exhibiting a topologically transitive attractor containing  periodic orbits (that are all hyperbolic) and one
singularity.  The examples in \cite{GW} are known as the geometric Lorenz attractors.

The Lorenz attractor is an example of a robustly non-hyperbolic star flow. This implies that the result in \cite{H} 
no longer holds for singular flows. In three dimensions, the difficulties introduced by the robust coexistence
of singularities and periodic orbits are almost fully understood. In particular, Morales et al., (see \cite{MPP} and \cite{GLW}) defined the notion of \emph{singular hyperbolicity}, which requires some compatibility
between the hyperbolicity of the singularity and that of the regular orbits. They prove that,
\begin{itemize}
 \item For  $C^1$-generic star flows on three-manifolds,
every chain recurrence class is singular hyperbolic.
\item Any robustly transitive set containing a singular point of a flow on a three-manifold is either a
singular hyperbolic attractor or a singular hyperbolic repeller.
\end{itemize}

It was conjectured in \cite{GWZ} that the same result holds without the generic assumption, that is, that  \emph{for any star flows on three-manifolds, every chain recurrence class is singular hyperbolic.} However, \cite{BaMo} built a star flow on a three-manifold demonstrating a chain recurrence class that is not singular hyperbolic, contradicting the conjecture.
This example exhibits two singularities of different indexes (dimension of the stable manifold) in the same chain
recurrence class. But the definition of singular hyperbolicity forces all singularities contained in a singular hyperbolic set
to exhibit the same index. However, these chain recurrence classes brake under perturbations, separating the
singularities of different indexes into different chain recurrence classes.

Several  generalizations to higher dimensions of results in  \cite{MPP} were attempted. The most general, using a direct generalization of singular hyperbolicity, is the result in \cite{GSW}. Here the authors proved that  an open and dense
set of star flows exist such that these star flows are singular hyperbolic if the singularities in the same chain recurrence class have the same index. However, the condition on the singularities, is not generic in higher dimensions. Open examples of star flows exist that do not satisfy it (see \cite{dL}). This made  a generalization of \cite{MPP} for an open and dense set in dimensions higher than four impossible.

In \cite{BdL}, the authors give a necessary and sufficient condition for a generic flow to be a star flow using a novel partial hyperbolicity called\emph{ multi-singular hyperbolicity}. This new hyperbolic structure is
compatible with singularities with different indexes, and the examples in \cite{BaMo} and \cite{BdL} are  multi-singular hyperbolic.
The following question appears in \cite{BdL} and \cite{CdLYZ}: 

\vskip 2mm
\emph{Question. Do all star flows  satisfy multi-singular hyperbolicity?}
\vskip 2mm


The first result gives a positive answer in three dimensions.
\begin{theorem}\label{T.main1}
Let $X$ be a star flow of a three-dimensional manifold $M$. Then, $X$ is multi-singular hyperbolic. 
\end{theorem}

The second result characterizes the singular hyperbolic star flows in three dimensions as follows:

\begin{theorem}\label{T.main2} 
Let $X$ be a star flow of a three-dimensional manifold $M$. If any chain recurrent class of $X$ has all its singularities with the same index, $X$ is singular hyperbolic. 
\end{theorem}

Finally, the last result contributes to our understanding of this question in higher dimensions: 

\begin{theorem}\label{T.main3} Let $X$ be a star flow of a  manifold $M$ and $C$ be a chain recurrence class of $X$. If $C$ is robustly chain transitive, then $X$ is multi-singular hyperbolic on $C$. 
\end{theorem}

This paper is organized as follows. Section~\ref{S.BasicDef} defines all objects and  provides the results needed
throughout this article. Section~\ref{SecNewVOT} proves that a vector field with all robustly periodic orbits of the same index is multi-singular hyperbolic, see Theorem \ref{T.main4}.
As  a consequence of Theorem \ref{T.main4},
Section~\ref{4} produces Theorems \ref{T.main1}, \ref{T.main2} and \ref{T.main3}.


\section{Basic definitions and preliminaries}\label{S.BasicDef}
Here, we formally define all objects that were informally discussed in the
introduction along with some other notions that will later be necessary.

Throughout this paper, $M$ is a compact boundaryless Riemannian manifold of dimension $d$. In section \ref{4}, we will deal
with three-dimensional manifolds in some results, but in those cases, it will be specified.
We will work with vector fields in $\cC^1$ topology and denote as $\varphi^t_X$ the flow associated with $X$.
We denote
the following sets of points of a vector field $X$: The set of singularities,  periodic orbits and critical elements,  $Sing(X)$, $Per(X)$ and  $Crit(X)=Sing(X)\cup Per(X)$ respectively.
Given an open set $U$, the maximal invariant set in $U$ is $$\Lambda_X=\bigcap_{t\in\RR}\varphi^t(U)\,.$$ In this article, we will focus on sets that exhibit some recurrence, called chain recurrence classes. For any vector
field $X$ of a manifold $M$, we can define the following notions that arise from Conley theory
 (see \cite{AN} and \cite{Co})
\begin{defi} Let $X$ be a  vector field of  $M$ and let $\varphi^t$ be its flow,
\begin{itemize}
\item We say that a pair of sequences $\set{x_i}_{0\leq i\leq k}$ and  $\set{t_i}_{0\leq i\leq k-1}$, $k\geq 1$, is an \emph{ $\varepsilon$-pseudo orbit from $x_0$ to $x_k$} for a flow $\varphi^t$,
  if for every $0\leq i \leq k-1$ one has
  $$ t_i-t_{i-1}\geq 1 \mbox{ and }d(x_{i+1},\varphi^{t_i}(x_i))<\varepsilon.$$
 \item A compact invariant set $\Lambda$ is called \emph{chain transitive} if  for any $\varepsilon > 0$, and any $x, y \in\Lambda$,
there is an $\varepsilon$-pseudo orbit from $x$ to $y$.  
\item A maximal invariant set $\Lambda_X$ is called \emph{robustly chain transitive} if there exists a compact neighborhood $U$ of $\Lambda_X$ and a $C^1$ neighborhood $\mathcal{U}$ of $X$  such that for all $Y\in\cU$, the maximal invariant set $\Lambda_Y$ in $U$ is the unique chain transitive set contained in $\text{Int}(U)$.
\item  We say that $x\in M$ is \emph{chain recurrent} if $\forall \, \varepsilon>0$, there is a nontrivial $\varepsilon$-pseudo orbit from $x$ to $x$. The set of chain recurrent points, is called the\emph{ chain recurrent set}
and is denoted by $\mathfrak{R}(M)$.
 \item We say that   $x, y \in \mathfrak{R}(M)$ are chain related if,  $\forall \,\varepsilon>0$, there are $\varepsilon$-pseudo orbits from $x$ to $y$ and from $y$ to $x$. This is an equivalence relation.
 The equivalence classes of this equivalence relation are called \emph{ chain recurrence classes}.
 \item A chain recurrence class $C$ is \emph{robustly chain transitive} if there exists a  neighborhood $U$ of $C$ in which the maximal invariant set $\Lambda_X$ in $U$ is such that $\Lambda_X=C$, and $\Lambda_X$ is robustly chain transitive.
\end{itemize}
\end{defi}
Now, we discuss the classical notions of hyperbolicity and the weaker versions. 
\begin{defi}
For $X\in \mathcal{X}^1(M)$ and a compact invariant set $\Lambda$ of $X$, we define the following notions.

\begin{itemize}
    \item $\Lambda$ is \textit{hyperbolic} if there are two constants $K\geq1$ and $\lambda>0$ and a continuous $D\varphi^t$-invariant splitting
$$T_{\Lambda}M=E^s\oplus <X(x)>\oplus E^u$$
such that for every $x\in\Lambda$ and $t\geq0,$ we have
$$||D\varphi^t|_{E^s(x)}||\leq K e^{-\lambda t},$$
$$||D\varphi^{-t}|_{E^u(x)}||\leq K e^{-\lambda t}.$$
Here $<X(x)>$ denotes the space spanned by $X(x)$, which is $0$-dimensional if $x$ is a singularity, or $1$-dimensional if $x$ is a regular point. 
\item $\Lambda$ has \textit{dominated splitting} of $X$, where $E$ is dominated by $F$ if there are constants $K\geq1$ and $\lambda>0$ and a continuous $D\varphi^t$-invariant splitting
$$T_{\Lambda}M=E\oplus_{\prec} F$$
such that for all $x\in \Lambda$ and $t\geq0$, we have
$$||D\varphi^t|_{E(x)}||\cdot||D\varphi^{-t}|_{F(\varphi^t(x))}||\leq Ke^{-\lambda t}.$$ 

\end{itemize}
\end{defi}

 The following definition is central to this article.
\begin{defi}
For $X\in \mathcal{X}^1(M)$, $X$ is a \emph{star vector field} and its flow is a \emph{star flow} if there exists a $C^1$ neighborhood $\mathcal{U}$ of $X$ such that every critical element (singularities or periodic orbits) of every $Y\in\mathcal{U}$ is hyperbolic.
\end{defi}
 The following definition was given in \cite{MPP} in dimension three and generalized by several authors, see for instance \cite{GSW}.
\begin{defi}
For $X\in \mathcal{X}^1(M)$ and a compact invariant set $\Lambda$ of $X$.
\begin{itemize}
\item $\Lambda$ is \emph{positively singular hyperbolic} for $X$ if there are constants $K\geq1$ and $\lambda>0$ and a continuous $D\varphi^t$-invariant splitting
$$T_{\Lambda}M=E^{ss}\oplus E^{cu}$$
such that for all $x\in \Lambda$ and $t\geq0$, the following conditions are satisfied
\begin{itemize}
    \item $E^{ss}$ is dominated by $E^{cu}$.
    \item $E^{ss}$ is uniformly contracting, i.e., $||D\varphi^t|_{E^{ss}(x)}||\leq K e^{-\lambda t}.$
    \item $E^{cu}$ is sectionally expanding, i.e., for any $2$ dimensional subspace $L\in E^{cu}(x)$, we have
    $$|det(D\varphi^t|_L)|\geq K^{-1}e^{\lambda t}.$$
\end{itemize}
\item We say that $\Lambda$ is \emph{negatively singular hyperbolic} for $X$ if $\Lambda$ is positively singular hyperbolic of $-X$. 
\item $\Lambda$ is \textit{singular hyperbolic} for $X$ if it is either: positively  or negatively singular hyperbolic of $X$, or a disjoint union of  positively and negatively singular hyperbolic sets of $X$.
\end{itemize}
\end{defi}

Multi-singular hyperbolicity  is a hyperbolic structure  that is related to the normal bundle and the extended linear Poincaré flow as originally  defined in \cite{GLW}.

To define multi-singular hyperbolicity we will need some preparations.  
First, we define the linear Poincaré flow as follows:

Let $X\in \cX^1(M)$ and $\varphi^t: M \rightarrow M$ its flow.
The normal space  of $X$ at $x$ is 
$$N_x=\{v \in  T_xM: v \perp X(x) \} .$$
Given $x$ a regular point and  $v \in N_x$, the differential of the flow induces a linear flow on $N=\bigcup_{x \in M \setminus Sing(X)} N_x ,$
 called the \emph{linear Poincaré flow}, is defined as follows:
$$\psi_t(v)=D\varphi_t(v)- \frac{\left \langle D\varphi_t(v), X(\varphi_t(x)) \right \rangle}{\left \|  X(\varphi_t(x))\right \|} X(\varphi_t(x)) ,$$
where $\left \langle.,. \right \rangle$ is the inner product on $T_xM$ given by the Riemannian metric.\\

To define the extended linear Poincaré flow, first, we define a set of directions over the singularities. 

\begin{defi}Let $M$ be a  $d$-dimensional manifold.
 \begin{itemize}
 \item We call \emph{the projective tangent bundle of $M$}  the fiber bundle  $\Pi_\PP\colon \PP M\to M$  whose fiber $\PP_x$ is
 the projective space of the tangent space $T_xM$: In other words, a point $L_x\in \PP_x$ is a $1$-dimensional vector subspace of $T_xM$.
 
  \item We call the \emph{normal bundle of $\PP M $}  the $(d-1)$-dimensional vector bundle over $\PP M$ $\Pi_\cN\colon \cN \to \PP M$, whose fiber $\cN_{L}$ over
  $L\in \PP_x$
  is the quotient space $T_x M/L$.

  If we endow $M$ with a Riemannian metric, $\cN_L$  is identified with the orthogonal hyperplane of $L$ in $T_xM$.
 \end{itemize}
 \end{defi}
 \begin{defi}
 Let $U$ be an open set of $M$, and $\Lambda_X\subset U$ be a maximal invariant set in $U$ for $\varphi^t$. Let $\sigma\in \Lambda_X\cap Sing(X)$ be a hyperbolic singularity and consider the finest dominated  splitting for $(D\varphi_t)_{t\in\RR}$:
 $$T_\sigma M=E^s_k\oplus_{\prec} \cdots\oplus_\prec E^s_1\oplus_\prec E^u_1\oplus_\prec\cdots\oplus_\prec  E^u_l.$$
 \begin{itemize}
\item Let $i$ be the smallest integer such that the strong stable manifold of $\sigma$ tangent to $E^s_k\oplus\cdots\oplus E^s_i$ intersects $\Lambda_X$
 only at $\sigma$. The space $E^{ss}_{\sigma,\Lambda_X}:=E^s_k\oplus\cdots\oplus E^s_i$ is called \emph{escaping stable space of $\sigma$ in $\Lambda_X$}.
\item  Analogously, we define the \emph{escaping unstable space of $\sigma$ in $\Lambda_X$}, and  denote it as $E^{uu}_{\sigma, \Lambda_X}:=E^s_j\oplus\cdots\oplus E^s_l$. 
\item The \emph{center space} at $\sigma$ is defined as 
 $E^c_{\sigma,\Lambda_X}=E^s_{i-1}\oplus \cdots \oplus E^s_1\oplus E^u_1\oplus \cdots\oplus E^u_{j-1}.$
 \item We denote by $\PP^c_{\sigma,\Lambda_X}$, the projective space of the center space $E^c_{\sigma,\Lambda_X}$. 
 \end{itemize}
 
\end{defi}
 Similar definitions can be made for singularities in other interesting sets (e.g., compact invariant set) and a compact neighborhood $U$. Here, an escaping stable space is a strong stable space such that all orbits (except the singularity) intersect the complement of $U$.
 
 \begin{defi}
 	Let $X$ be a vector field in $\cX^1(M)$ and $\Lambda$ be a compact invariant  set such that its singularities are hyperbolic, we define:
 	\begin{itemize}
 	    \item The lift function $S:M\setminus Sing(X)\to\PP M$  by $$S(x)= <X(x)>$$
 	    \item Then, the lift of $\Lambda$ is given by,
 	    $$\widetilde{\Lambda}=\overline{\big\{<X(x)>: x\in \Lambda\setminus Sing(X) \big\}}$$ or equivalently $\widetilde{\Lambda}=\overline{S(\Lambda\setminus Sing(X))}$
 	    \end{itemize}
 \end{defi}

  \begin{defi}
  \begin{itemize}
  	Let $X$ be a vector field in $\cX^1(M)$. Let $U\subset M$ be an open set and  $\Lambda$ be the maximal invariant  set in $U$ such that its singularities are hyperbolic.
 	    \item The \emph{extended invariant set of $\Lambda$} is the compact subset of $\PP M$ defined by
 	$$B(X,\Lambda)=\widetilde{\Lambda}\cup\bigcup_{\sigma\in Sing(X)\cap \Lambda} \PP^c_{\sigma,\Lambda}.$$
 	 \item 	Conley theory asserts that  any chain recurrent  class $C$ admits a basis of neighborhoods which are nested filtrating neighborhoods $U_{n+1}\subset U_n$, $C=\bigcap U_n$. Let us consider the maximal invariant sets $\Lambda_n\subset U_n$,
 we define
$$B(X,C)=\bigcap_n\widetilde{B(X,\Lambda_n)}.$$
\end{itemize}
\end{defi}
 In the same manner, we can define the notion of $B(X,K,U)$ for a compact invariant set $K$ and a compact neighborhood $U$, as the lift of $K$ and the center space of $U$.

 \begin{defi} For $X$ a vector field in $\cX^1(M)$ and let $\Lambda$  be a compact invariant  set.

 \begin{itemize}
  \item Let  $\varphi_{\PP}^t\colon\PP M\to \PP M$ be the flow defined by $$\varphi_{\PP}^t(L_x)= D\varphi_x^t(L_x)\in \PP_{\varphi^t(x)}M.$$
 \item Let $\psi_{\cN}^t\colon\cN\to \cN$ be the extended linear Poincaré flow whose restriction to a fiber $\cN_{L_x}$, $L_x\in \PP_x M$,
  is the linear isomorphism onto
  $\cN_{\varphi^t_{\PP}(L_x)}$ defined as follows: $D\varphi^t_x$ is a linear isomorphism from $T_xM$ to $T_{\varphi^t(x)}M$, which maps the line $L_x\subset T_xM$
  onto the line $\varphi^t_{\PP}(L_x)\subset T_{\varphi^t(x)};$.  Therefore, it passes to the quotient as the announced linear isomorphism.
  $$\begin{array}[c]{ccc}
T_xM &\stackrel{D\varphi^t_x}{\longrightarrow}&T_{\varphi^t(x)}M\\
\downarrow&&\downarrow\\
\cN_{L_x}&\stackrel{\psi^t_{\cN}}{\longrightarrow}&\cN_{\varphi^t_{\PP}(L_x)}
\end{array}$$
Suppose $M$ is endowed with a Riemannian metric, then  $\cN_{L_x}$ is identified with the orthogonal hyperplane of $L_x \in T_xM$.
 \end{itemize}
\end{defi}

However, looking at the extended linear Poincare flow, makes us lose some of the information about expansion
rates along the flow direction, that is unimportant away from singularities but, in our case, plays a crucial
role in this study. To recover this information, we  need a multiplicative cocycle, as the ones  defined in\cite{BdL}.

 Let $C$ be a chain recurrence class of $X$. For every singularity, $\sigma\in Sing (X) \cap C$, we consider a neighborhood $U_{\sigma}$ of $\sigma$, such that $\{\sigma\} $ is the maximal invariant set in it. We consider a Riemannian metric
 $\| \cdot \|$ as well, such that  $\|(D_x\varphi^t)|_{L_x}\|=1$  in the complement of $\bigcup_{\sigma\in  Sing (X)} U_{\sigma}$.
\begin{coro}[Corollary 1 in \cite{BdL}]
For all $\sigma\in Sing (X)$, there exists a multiplicative cocycle $$h_\sigma: B(X,\Lambda)\times \RR\to \RR$$  such that  if $x$ and $\varphi^t(x)$
are in $U_{\sigma}$, then $ h^t_{\sigma}(L_x)= \|(D_x\varphi^t)|_{L_x}\|$ and if $x$ and $\varphi^t(x)$ are not in $U_{\sigma}$, then    $ h^t_\sigma(L_x)=1$.
\end{coro}

\begin{defi}
For $X \in \mathcal{X}^1(M)$ and $\Lambda$, a compact invariant set.
\begin{itemize}
\item  We say that $X$ is  \emph{multi-singular hyperbolic} in $\Lambda$ if there exists an open set  $U$ with $\Lambda \subset U$, such  that
\begin{itemize} 
\item All singularities in $U$ are hyperbolic
\item There is an invariant continuous dominated splitting $\cN=E\oplus_{\prec} F$ for  $\psi^t_\cN$   over $B(X,\Lambda)$
\item There are sets of singularities $ S_-\subset \Lambda \cap Sing (X) $, $ S_+ \subset \Lambda \cap Sing (X) $,  so that the vectors in $E$ are uniformly contracted by the flow $$h_-^t\cdot \psi^t_\cN$$  and the ones in $F$ are uniformly expanded by the flow  
 $$h_+^t\cdot \psi^t_\cN\,.$$
 \end{itemize}
 \item A multi-singular hyperbolic flow is a flow that is multi-singular hyperbolic in all its chain recurrence classes.
\end{itemize}
\end{defi}
Note that all multi-singular hyperbolic flows are star flows.
Now, we
 state some classical results that will be used in the following sections. 
 
 The following Lemma needs this hypothesis:


\emph{(A): For any  $T> 0$, the set of periodic orbits of period $T$ is isolated.}
 
\begin{lemm}\label{ConConLem} \cite{C}
Let $X \in \mathcal{X}^1(M)$ be a vector field that satisfies condition $A$, $\cU$ a neighborhood of $X$ in $\mathcal{X}^1(M)$,  and $C$ a chain recurrence class of $X$. Then, for any $\varepsilon>0$, there exist a vector field $Y_{\varepsilon}\subset \cU$ and a critical element $\gamma_{\varepsilon}$ of $Y_{\varepsilon}$  that is $\varepsilon$ close to $C$ in the Hausdorff topology.
\end{lemm}

\begin{defi}
Let $x \in M \setminus Sing(X)$. We say that $x$ is strongly closable for $X$, if for any $C^1$ neighborhood $\cU$ of $X$, and any $\delta>0$, there exists $Y \in \cU$, $y \in M$, and $\tau >0$, such that the following items are satisfied:
\begin{itemize}
    \item $\varphi_{\tau}^Y(y)=y$
    \item $d(\varphi_t^X(x), \varphi_t^Y(y))<\delta$, for any $0\leq t \leq \tau$.
\end{itemize}
The set of strongly closable points of $X$ is denote by $\Sigma(X)$.
\end{defi}

The following is the ergodic closing Lemma of Ma\~n\'e.
\begin{lemm}\label{Erg}(\cite{M1}\cite{W1})
For any $X \in \mathcal{X}^1(M)$, $\mu (Sing(X)\cup \Sigma(X))=1$ for every $T>0$ and every $\varphi_X^T$-invariant Borel probability measure $\mu$.
\end{lemm}
\begin{theo}\label{Kupka-Smale}(Kupka-Smale \cite{PadM}\cite{S1}\cite{K})
The set consisting  of vector fields $X \in \mathcal{X}^r(M)$ with  $r\geq1$,  such that
all critical elements are hyperbolic is residual (and therefore dense) in $\mathcal{X}^r(M)$.
\end{theo}



\begin{defi}
Let $X \in \mathcal{X}^1(M)$ be a vector field and  $\sigma$ be a hyperbolic singularity of $X$. If $\lambda_1 \leq \cdots  \lambda_s<0< \lambda_{s+1}\leq \cdots \leq \lambda_d$ are the Lyapunov exponents of $D\varphi^t(\sigma)$, then the saddle value \textsl{sv}$(\sigma)$ of $\sigma$ is defined as
$$
\textsl{sv}(\sigma)=\lambda_s+\lambda_{s+1}.$$  

\end{defi}

\begin{rema}
For the Lorenz attractor, the saddle value  $\textsl{sv}(\sigma)$ of $\sigma$ is positive. Also, for the Lorenz repeller, the saddle
value $\textsl{sv}(\sigma)$ of $\sigma$ is negative.
 
\end{rema}
Let $S_-$ and $S_+$  denote the sets of singularities that exhibit positive  and negative saddle values,
respectively. And
$$h_-^t=\prod_{\sigma_i\subset S_{-}}h_{\sigma_i}^t \text{ and } h_+^t=\prod_{\sigma_i\subset S_{+}}h_{\sigma_i}^t.$$



\begin{defi}
Let $X \in \mathcal{X}^1(M^d)$ and $\sigma$ be a hyperbolic singularity of $X$. Assume that $C(\sigma)$ is nontrivial and the Lyapunov exponents of $D\varphi^t(\sigma)$ are $\lambda_1 \leq \cdots  \lambda_s<0< \lambda_{s+1}\leq \cdots \leq \lambda_d$. We say that $\sigma$ is Lorenz-like if the following conditions are satisfied:
\begin{itemize}
\item $\textsl{sv}(\sigma)\neq 0$.

\item If $\textsl{sv}(\sigma)>0$, then $\lambda_{s-1}<\lambda_s$, and $W^{ss}(\sigma) \cap C(\sigma)=\{\sigma\}$. Here $W^{ss}(\sigma)$ is the invariant manifold corresponding to the bundle $E^{ss}_{\sigma}$ of the  splitting $T_{\sigma}M=E^{ss}_{\sigma} \oplus E^{cu}_{\sigma}$, where $E^{ss}_{\sigma}$ is the invariant space corresponding to the Lyapunov exponents $\lambda_1, \lambda_2, \cdots, \lambda_{s-1}$ and $E^{cu}_{\sigma}$ corresponding to the Lyapunov exponents $\lambda_s, \lambda_{s+1}, \dots, \lambda_d$.

\item If $\textsl{sv}(\sigma)<0$, then $\lambda_{s+1}<\lambda_{s+2}$, and $W^{uu}(\sigma) \cap C(\sigma)=\{\sigma\}$. Here, $W^{uu}(\sigma)$ is the invariant manifold corresponding to the bundle $E^{uu}_{\sigma}$ of the  splitting $T_{\sigma}M=E^{cs}_{\sigma} \oplus E^{uu}_{\sigma}$, where $E^{cs}_{\sigma}$ is the invariant space corresponding to the Lyapunov exponents $\lambda_1, \lambda_2, \cdots, \lambda_{s+1}$ and $E^{uu}_{\sigma}$ corresponding to the Lyapunov exponents $\lambda_{s+2}, \lambda_{s+1}, \dots, \lambda_d$.

\end{itemize}
\end{defi}

\begin{defi}
Let $X\in\mathcal{X}^1(M)$ be such that there exists  $\sigma\in Sing(X)$ with a non-zero saddle value. Let  $s$ be the stable index of $\sigma$. Then, the periodic index of $\sigma$, denoted by $Ind_p(\sigma)$, is defined by
$$Ind_p(\sigma)=\left\{\begin{array}{ll}
    s & \text{if } \textsl{sv}(\sigma)<0, \\
    s-1 & \text{if } \textsl{sv}(\sigma)>0.
\end{array}\right.$$
For a periodic orbit $p$ of $X$, we define $Ind_p(p)=Ind(p)$.
\end{defi}

\begin{defi}
\begin{itemize} For a vector field  $X\in\mathcal{X}^1(M)$, 
    \item We say that $X$ is homogeneous in an open set $U\subset M$ if  there exists a $C^1$ neighborhood $\cU$ of $X$, such that there exist $Y_p\in\cU$ with a critical element in $U$, and  all critical elements for all $Y\in\cU$ in $U$   have the same periodic index.
    \item We say that $X$ is homogeneous if, for every chain recurrence class of $X$ there exists an open set  in which $X$ is homogeneous.
\end{itemize}
\end{defi}

\begin{rema}
\begin{itemize}
\item Multi-singular hyperbolic vector fields are homogeneous.
\item It follows from Theorem~\ref{Kupka-Smale}, that if $X$ is a homogeneous vector field then it is a star flow.
\end{itemize}
\end{rema}


\section{Homogeneity and multi-singular hyperbolicity }\label{SecNewVOT}
The aim of this section is to prove the following Theorem:
\begin{theo}\label{T.main4} For a vector field  $X\in\mathcal{X}^1(M)$,  and $C$ a chain recurrence class of $X$, if there is a neighborhood $U$ of $C$ in which $X$ is homogeneous, then $X$ is multi-singular hyperbolic on $C$.
\end{theo}

Our other theorems will be derived from this one.

The methodology in this paper consist of \emph{cleaning} the use of some generic hypotheses found in the previous references.
Rather than stating theorems as ``\emph{there exists a generic set in which all $X$....}",
we state them with the explicit minimal generic hypothesis needed.
In several articles cited here, for instance, Lemma~\ref{ConConLem} is cited as a generic lemma, and this is true if the set of vector fields we considered was $\mathcal{X}^1(M)$. However, for the set of star vector fields, all vector fields satisfy condition $A$ and therefore, restricting ourselves to a generic subset of star vector fields is not necessary.
After this, the rest of the generic hypotheses used to show that a star flow is multi-singular hyperbolic are
used to prove that generic star flows are homogeneous. Since multi-singular hyperbolic vector fields are always
homogeneous, it is a necessary hypothesis. Thus, we restate the theorems using the homogeneous hypothesis
instead of $X$ being generic. Furthermore, this is  useful since 
instances exist where we can ensure that all star flows are homogeneous, for instance three dimensions. 

\begin{lemm}[\cite{GY}]\label{GY}
Let $X\in\mathcal{X}^1(M)$ be a vector field and $\sigma\in Sing(X)$ be hyperbolic, there exists a neighborhood $U$ of $\sigma$ such that there is no contracting   periodic point in $U$ for the extended linear Poincar\'e flow. 
\end{lemm}
\begin{rema}\label{sillas}
From Lemma \ref{ConConLem}, for any star vector field $X\in\mathcal{X}^1(M)$ and a nontrivial chain recurrence class $C(\sigma)$, where $\sigma\in Sing(X)$, there is a sequence of star vector fields $Y_n\to X$ and a sequence of periodic orbits $\gamma_n$ that converge to $C(\sigma)$ in the Hausdorff topology. Consequently, from Lemma~\ref{GY}, we  conclude that these periodic orbits are not contracting or expanding.
\end{rema}
\begin{lemm}[\cite{L}\cite{Ma2}]\label{l.uniformlyattheperiod}
For any star vector field $X$ on a compact  manifold $M$, there is a $C^1$ neighborhood $\cU$ of $X$ and numbers $\eta> 0$
and $T > 0$ such that, for any periodic orbit $\gamma$ of a vector field $Y\in \cU$ and any integer $m>0$; the following holds: Let  $N= N_s\oplus N_u$ be
the stable- unstable splitting of the normal bundle $N$ for the linear Poincar\'e flow $\psi_t^Y$  Then:
\begin{itemize}
\item \emph{Domination: }For every $x \in \gamma$ and $t \geq T$, one has  $$\frac{\norm{\psi_t^Y\mid_{N_s}}}{\min(\psi_t^Y\mid_{N_u})}\leq e^{-2\eta t}\,.$$
\item\emph{ Uniform hyperbolicity at the period:} If the period $\pi(\gamma)$ is larger than  $T$, then, for every $x\in \gamma$,  one has the following:

$$
    \prod^{[m\pi(\gamma)/T]-1}_{i=0}\norm{\psi_t^Y\mid_{N_s}(\varphi^Y_{iT}(x))}\leq e^{-m\eta \pi(\gamma)}
$$
and
$$
    \prod^{[m\pi(\gamma)/T]-1}_{i=0}\min(\psi_t^Y\mid_{N_u}(\varphi^Y_{iT}(x)))\geq e^{m\eta \pi(\gamma)}\,.$$
    Here, $\min(A)$ is the mini-norm of $A$, i.e., $\min(A) = \norm{A^{-1}}^{-1}$.
\end{itemize}
\end{lemm}

\begin{rema}
Under the same assumptions of the previous lemma, the uniform hyperbolicity at the period is sometimes expressed in the following equivalent form: $$
    \sum^{[m\pi(\gamma)/T]-1}_{i=0}\log\norm{\psi_t^Y\mid_{N_s}(\varphi^Y_{iT}(x))}\leq -m\eta \pi(\gamma)$$
\end{rema}
The following two Lemmas from \cite{GSW}  imply that star flows exhibit Lorenz-like singularities.

\begin{lemm}[\label{4.2}Lemma 4.2 in \cite{GSW}]
Let $X $ be a star flow in $M$ and $\sigma\in Sing(X)$. Assume that the Lyapunov exponents of
$D\varphi_t(\sigma)$ are $$\lambda_1\leq\dots\leq\lambda_{s-1}\leq\lambda_s <0<\lambda_{s+1}\leq\lambda_{s+1}\leq\dots\leq \lambda_d\,.$$

If the chain recurrence class $C(\sigma)$ of $\sigma$
is nontrivial, then
\begin{itemize}
\item either $\lambda_{s-1}\neq\lambda_s$ or $\lambda_{s+1}\neq\lambda_{s+2}$.
\item If  $\lambda_{s-1}=\lambda_s$, then $\lambda_s +\lambda_{s+1}<0$.
\item If $\lambda_{s+1}=\lambda_{s+2}$, $\lambda_s +\lambda_{s+1}>0$.
\item  If $\lambda_{s-1}\neq\lambda_s$ and $\lambda_{s+1}\neq\lambda_{s+2}$, then $\lambda_s +\lambda_{s+1}\neq0$.
\end{itemize}
\end{lemm}

\begin{lemm}[\label{4.7}Lemma 4.7 in \cite{GSW}]
Let $X$ be a star flow in $M$ and $\sigma$ be a singularity of $X$ such that $C(\sigma)$ is nontrivial.
Then for any singularity $\rho\in C(\sigma)$, we have the following:
\begin{itemize} \item if $sv(\rho)>0$,  one has $W^{ss}(\rho) \cap C(\sigma)=\{\rho\}.\,$ Here $W^{ss}(\rho)$ is the invariant manifold corresponding to the bundle $E^{ss}_{\rho}$ of the  splitting $T_{\rho}M=E^{ss}_{\rho} \oplus E^{cu}_{\rho}$, where $E^{ss}_{\rho}$ is the invariant space corresponding to the Lyapunov exponents $\lambda_1, \lambda_2, \cdots, \lambda_{s-1}$ and $E^{cu}_{\rho}$ corresponding to the Lyapunov exponents $\lambda_s, \lambda_{s+1}, \dots, \lambda_d$.

\item If $sv(\rho)<0$, one has $W^{uu}(\rho) \cap C(\sigma)=\{\rho\}.\,$ Here $W^{uu}(\rho)$ is the invariant manifold corresponding to the bundle $E^{uu}_{\rho}$ of the  splitting $T_{\rho}M=E^{cs}_{\rho} \oplus E^{uu}_{\rho}$, where $E^{cs}_{\rho}$ is the invariant space corresponding to the Lyapunov exponents $\lambda_1, \lambda_2, \cdots, \lambda_{s+1}$ and $E^{uu}_{\rho}$ corresponding to the Lyapunov exponents $\lambda_{s+2}, \lambda_{s+1}, \dots, \lambda_d$.

\end{itemize}
\end{lemm}
To summarize both results, we present the following corollary:
\begin{coro}\label{c.lorenzlike}
Let $X$ be a star flow in $M$ and $\sigma$ be a singularity of $X$ such that $C(\sigma)$ is nontrivial.
Then, any singularity in $C(\sigma)$ is Lorenz-like.
\end{coro}

The following lemma is a restatement of  Corollary 66 in \cite{BdL}. 
 	Let $X\in\cX^1(M)$, let $\sigma$ be a Lorenz-like singularity with splitting $T_\sigma M=E^{s}\oplus E^c\oplus E^{uu}$,
 	and let $\PP E^{cs}_\sigma$ denote the projective space of $E^s\oplus E^c$. Then, the extended linear Poincar\'e flow admits a dominated splitting $\cN^s\oplus \cN^u$
 	with $\dim(\cN^s)=\dim(E^{s})$ over $\PP E^{cs}_\sigma$.

\begin{lemm}~\label{lorenzlike-cocycle}
 For $X\in\cX^1(M)$, let $\sigma$ be a Lorenz-like singularity with splitting $T_\sigma M=E_{\sigma}^{cs}\oplus E_{\sigma}^{uu}$
 	and let $\PP E^{cs}_{\sigma}$ denote the projective space of $E_{{\sigma}}^{cs}$, for which the extended linear Poincar\'e flow admits a dominated splitting $\cN_{L_{\sigma}}=E\oplus E_{\sigma}^{uu}$,
  for ${L_{\sigma}}\in \PP E^{cs}_\sigma$, 	with $\dim(E)=s$ where $s$ is the stable index of $\sigma$.
 Then,  $(h^t_{-}\cdot\psi_t|_E)$ contracts uniformly.
 \end{lemm}

Let us observe that the singularities in this lemma have negative saddle value, and the dimension of the stable manifold of the singularity is $\dim(E^{cs})-1$, in this splitting.  
An analogous statement can be made for Lorenz-like singularities with positive saddle value  and   $(h^t_{+}\cdot\psi_t|_{F})$.

A similar version of the following lemma is found in \cite{BdL}.
Since the hypotheses are stated differently and this produces a small effect on the proof, we re-do the proof here.  

\begin{lemm}\label{faltalabel}
For $X\in\cX^1(M)$, let $C$ be a nontrivial chain recurrence class. Let $\cU$ be a neighborhood of $X$ and $U$ a neighborhood of $C$  in which $X$ is homogeneous. Then, by reducing $U$ if necessary, the lifted maximal invariant set $B(X,\Lambda)$ of $X$ in $U$ has a dominated splitting $\cN=E\oplus_{\prec} F$ for the
extended linear Poincar\'e flow. Thus, $E$ extends the stable bundle for every periodic orbit $\gamma\subset U$ of any $Y\in\cU$.
\end{lemm}

\begin{proof}

From Lemma~\ref{ConConLem}, for every $C\in U$ and every $n$, there are vector fields $Y_n\in\cU$  and periodic orbits $\gamma_n\in U$ such that $\gamma_n\to_{H} C$ where $.\to_{H}$ denotes the convergence in the Hausdorff topology. 
The homogeneous property  asserts that all $\gamma_n$  are hyperbolic and of the same periodic index. Thus, the periodic orbits $L_n=\set{<Y_n(x)>\,\text{ such that } x\in\gamma_n}$  are hyperbolic for the extended linear Poincar\'e flow with the same index for all $n$.  Therefore,  they induce a dominated splitting in their closure with the same index.  Now, by reducing $U$ if necessary, and letting  $\Lambda$ be the maximal invariant set in $U$, this domination extends to  $$\widetilde{\Lambda_r}=\overline{\big\{<X(x)>: x\in \Lambda\cap \mathfrak{R}(M)\setminus Sing(x)\big\}}\,.$$
If $C$ is not singular, we are done, and this completes the proof.

Suppose that $C$ exhibits singularities.If $\sigma$ is a singularity with negative saddle value, Lemma~\ref{lorenzlike-cocycle} implies that there exists a dominated splitting for the extended linear Poincaré flow $\cN_{L_{\sigma}}=E\oplus E_{\sigma}^{uu}$, over $L_{\sigma}\in \PP_{\sigma} E^{cs}$. Since the periodic index of the singularities is the same as the index of periodic orbits, the splitting over $\widetilde{\Lambda_r}$ and  $L_{\sigma}\in \PP_{\sigma} E^{cs}$ exhibit the same index.

The intersection of both spaces is nonempty since $C$ is nontrivial. This is because there are regular orbits approximating the singularity accumulating in its stable manifold, and the same is true for the unstable manifold.
 The uniqueness of the dominated splittings (see \cite{BDV}) for a prescribed dimensions implies that these dominated splittings coincide at the intersection.
Analogous reasoning applies to the singularities of positive saddle value. Thus, considering all singularities, there is a
dominated splitting of index $s$ over
$$\widetilde{\Lambda_r}\cup\bigcup_{\sigma\in S_{-}\cap\Lambda}\PP_{\sigma} E^{cs}\cup\bigcup_{\sigma\in S_{+}\cap\Lambda}\PP_{\sigma} E^{cu}\,.$$
Since $$B(X,C)\subset \overline{\big\{<X(x)>: x\in \Lambda\cap\mathfrak{R}(M)\setminus Sing(X)\big\}}\bigcup_{\sigma\in 
S_{-}\cap\Lambda}\PP_{\sigma} E^{cs}\bigcup_{\sigma\in S_{+}\cap\Lambda}\PP_{\sigma} E^{cu}\,$$

Then $B(X,C)$ has a dominated splitting of index $s$.

\end{proof}

From Remark~\ref{sillas} we have that $s>0$ 

The following Theorem from \cite{GSW} describes all ergodic measures for a star flow. 
\begin{theo}[\label{5.6}Lemma  5.6 \cite{GSW}]
Let $X\in\mathcal{X}^1(M)$ be a star flow. Any invariant ergodic measure $\mu$ of the flow $\phi^t$ is a hyperbolic measure. Moreover, if $\mu$ is not the atomic measure
on any singularity, then $supp(\mu) \cap H(P)\neq \emptyset$, where $P$ is a periodic orbit with the index of $\mu$, i.e., the number of negative Lyapunov
exponents of $\mu$(with multiplicity).
\end{theo}

We  restate another Lemma  (73) from \cite{BdL}, changing the hypothesis of being  \emph{generic}  for the concrete generic condition, which is being homogeneous. Since this Lemma is fundamental for Theorem~\ref{T.main4} we will re-do the proof here, to show that no other hypothesis is needed.
Additionally, we restate the claims in the proof as separated technical propositions. If no generic hypothesis is used, we will not provide a proof.

\begin{prop}[\label{Claim1} Claim 1, Lemma 73 in  \cite{BdL}]Let $X\in\mathcal{X}^1 (M)$ and a chain recurrent class $C$ of $X$.
Suppose that for any neighborhood $U$ of $C$, $B(X,\Lambda)$ is  such that the normal space has a dominated  splitting $\cN_{B(X,\Lambda)}=E\oplus_{\prec}F$ but the space $E$ is not uniformly contracting for the flow  $h_{-}^t\cdot\psi^t_{\cN}$
Then, for every $T>0$, there
exists an ergodic invariant measure $\mu'_T$ whose support is contained in $$B(X,C)=\bigcup_{\sigma\in Sing(X)\cap C}\PP^c_s\cup \tilde C$$ such that
$$\int
\log\norm{h^T_{-}.\psi^T_{\cN} \mid_{E}} d\mu'_T(x)\geq 0\,,$$
where $\tilde C=\overline{S(C\setminus Sing(X))}$. 

 \end{prop}
  Here, the idea is that since $E$ is not uniformly contracting, $E$ is not uniformly contracting,Then from Lemma I.5 in \cite{Ma2}  then there is an ergodic measure with support in over $B(X,\Lambda)$ such that $$\int
\log\norm{h^T_{-}.\psi^T_{\cN} \mid_{E}} d\mu(x)\geq 0\,.$$
However, this measure may not be supported in $\bigcup_{s\in Sing(x)\cap C}\PP^c_s\cup \tilde C$. Since this holds for any neighborhood and there exist neighborhoods $U_n\to C$, this defines a sequence of ergodic measures with the property above and converging to an invariant measure with support in $\bigcup_{\sigma\in Sing(X)\cap C}\PP^c_s\cup \tilde C$. Therefore, we can decompose  this measure as a sum of ergodic measures, which proves the proposition. 

The following proposition shows that $h^T_{-}$ does not affect the hyperbolicity of periodic orbits.

\begin{prop}[\label{Claim2}Claim 2, Lemma 73 in  \cite{BdL}] For $X\in\mathcal{X}^1(M)$, let $\nu_n$ be a measure supported on a  periodic orbit $\gamma_n$ with period $\pi( \gamma_n) >T$. Suppose $S$ is any subset of $Sing(X)$ and  $h^t_{S}=\prod_{\sigma\in S}h^t_{\sigma}$. Then $\int\log h^T_{-} d\nu_n(x)=0$.
\end{prop}

In the following claim in \cite{BdL} there is a somewhat confusing use of the generic hypotheses. We rewrite the statement and the proof here.

\begin{prop}[\label{Claim3}Claim 3, Lemma 73 in  \cite{BdL}]
Let $X\in\mathcal{X}^1 (M)$ and $C$ be a chain recurrent class of $X$.
Let $\cU$ be a neighborhood of $X$ and $U$ a neighborhood of $C$ in which $X$ is homogeneous in $C$. 
Suppose that for every $T>0$, there
exists an ergodic invariant measure $\mu'_T$ whose support is contained in $$\bigcup_{\sigma\in Sing(X)\cap C}\PP^c_{\sigma}\cup \tilde C$$ such that
$$\int
\log\norm{h^T_{-}.\psi^T_{\cN} \mid_{E}} d\mu(x)'_T\geq 0\,,$$
where $\tilde C=\overline{S(C\setminus Sing(X))}$. 
Then, there is a singular point $\sigma_i\subset C$ such that $\mu_T$ is supported on $\PP^c_{\sigma_i}$.

\end{prop}
\begin{proof}
We prove this by contradiction. Suppose that  $$\mu_T(\bigcup_{\sigma_i\in Sing(X)}\PP^c_{\sigma_i})=0\,.$$
Since $S|_{C/Sing(X)}$ is a $C^1$ diffeomorphism, then it projects on $M$ to an ergodic measure $\nu$ supported on the class $C$ such that it weighs zero i in the singularities. For $\nu$, we have that $$\int \log\norm{h^T_{-}.\psi^T \mid_{E}} d\nu(x)\geq 0 \,.$$

Recall that $\psi^T $ is the linear Poincar\'{e} flow defined over the regular points of $M$, and $h^T_{-}$ can be defined as a function of $x\in M$ rather than a function of $L\in\PP M$ outside an arbitrarily small neighborhood of the singularities.

Now, Birkhoff theorem asserts that for $a$ in a $\nu$-full measure set of points  and a time $T>0$, $$\lim_{n\to\infty}\frac{1}{n}\sum_{i=0}^{n-1}\log\norm{h^{iT}_{-}.\psi^{iT} \mid_{E_{\varphi^{(i-1)T}(a)}}}=\int \log\norm{h^{T}_{-}.\psi^T \mid_{E}} d\nu(x)\,.$$
Then 
$$\lim_{n\to\infty}\frac{1}{n}\sum_{i=0}^{n-1}\log\norm{h^{iT}_{-}.\psi^{iT} \mid_{E_{\varphi^{(i-1)T}(a)}}}\geq0\,.$$
The homogeneous property implies that all periodic orbits must contract $E$ at the period, therefore, $a$ cannot be periodic.

Since $\nu$ weighs zero on the singularities, the ergodic closing Lemma~\ref{Erg} implies that $\nu(\Sigma(X)\cap C)=1$, where  $\Sigma(X)$ is the set of strongly closable points. 

For a closable $a$ satisfying the Birkhoff theorem, there is a sequence vector fields $Y_m$ converging to $X$ with a sequence of periodic points $b_m$ belonging to periodic orbits $\gamma_m$, such that $b_m\to a$ and $\gamma_m\to \mathcal{O}(a)$, where $\mathcal{O}(a)$ is the orbit of $a$.
Since $a$ is not periodic, the periods of $b_m$ tend to infinity as they approach $a$.
As a consequence, there exists $m_0$ such that if $m>m_0$, 
$$\lim_{n\to\infty}\frac{1}{n}\sum_{i=0}^{n-1}\log\norm{h^{iT}_{-}.\psi^{iT} \mid_{E_{\varphi_{Y_n}^{(i-1)T}(b_m)}}}\geq0\,.$$
 This implies that for $Y_m$ with $m>m_0$, there exists an ergodic measure $\nu_m$ supported in $\gamma_m$ such that
$$\int \log\norm{h^T_{-}.\psi^T\mid_{E}} d\nu_m(x)\geq 0 \,.$$

Since we are over a periodic orbit, Proposition \ref{Claim2} gives 
$\int\log h^T_{-} d\nu_m(x)=0$, and then

$$\int \log\norm{\psi^T \mid_{E}} d\nu_m(x)\geq 0 \,.$$
This contradicts  Lemma \ref{faltalabel} since  for $m>m_0$ large enough  $Y_m\in\cU$ and therefore  $N^s_{\gamma_m}=E$.

\end{proof}

The following Lemma  directly implies Theorem~\ref{T.main4}
 \begin{lemm}[\label{l.contraction} Lemma 7.3 in  \cite{BdL}]
For $X\in\mathcal{X}^1 (M)$ let $C$ be a chain recurrent class of $X$ and let $\cU$ be a neighborhood of $X$ and $U$ a neighborhood of $C$ in which $X$ is homogeneous in $C$.
Then, by reducing $U$ if necessary, the extended maximal invariant set $B(X,\Lambda)$ in $U$  is such that the normal space has a dominated  splitting $\cN_{B(X,\Lambda)}=E\oplus_{\prec}F$ where the space $E$ (resp. $F$) is uniformly contracting (resp. expanding) for the flow  $h_{-}^t\cdot\psi^t_{\cN}$
(resp. $h_{+}^t\cdot\psi^t_{\cN}$).
\end{lemm}
\begin{proof}
Suppose  by contradiction that for any neighborhood $U$ of $C$, $B(X,\Lambda)$ is  such that the normal space exhibits a dominated  splitting $\cN_{B(X,\Lambda)}=E\oplus_{\prec}F$, and the space $E$ is not uniformly contracting for the flow  $h_{-}^t\cdot\psi^t_{\cN}$.
From Proposition~\ref{Claim1}, it follows that,
 for every $T>0$, there
exists an ergodic invariant measure $\mu_T$ whose support is contained in $\bigcup_{s\in Sing(x)\cap C}\PP^c_s\cup \tilde C$ such that
$$\int
\log\norm{h^T_{-}.\psi^T_{\cN} \mid_{E}} d\mu_T\geq 0\,,$$
where $\tilde C=\overline{S(C\setminus Sing(x))}$. 

 Proposition~\ref{Claim3} implies that   there is a singular point $\sigma_i\subset C$ so that $\mu_T$ is supported on $\PP^c_{\sigma_i}$. However, this contradicts Lemma~\ref{c.lorenzlike}.
\end{proof}
\section{Consequences of Theorem~\ref{T.main4}}\label{4}
The following corollary is a direct consequence of  Theorem~\ref{T.main4}
\begin{coro}\label{c1Main4}
Let $X\in\mathcal{X}^1(M)$ be homogeneous. Then, $X$ is multi-singular hyperbolic.
\end{coro}

The following lemma along with Theorem~\ref{T.main4} implies Theorem~\ref{T.main1}
\begin{lemm}\label{l.threehom}
Let $M$ be a three-dimensional manifold and let $X\in\mathcal{X}^1(M)$ be a star flow. Then, $X$ is homogeneous.
\end{lemm}
\begin{proof}
Suppose that there is a chain recurrence class $C$  consisting only of one hyperbolic critical element. Here, trivially, all critical elements in the class and in a neighborhood of the class exhibit the same index.
This holds for any small perturbation of $X$, so $X$ in the neighborhoods of these chain recurrence classes is homogeneous.

If there is a singularity $\sigma \in C$, only four possible hyperbolic splittings exists  for $\sigma$ in  three dimensions.
\begin{itemize}
    \item $T_{\sigma}M=E^s\oplus E^u$ where $E^s$ is two-dimensional, and $E^u$ is one-dimensional.
     \item $T_{\sigma}M=E^s\oplus E^u$ where $E^u$ is two-dimensional, and $E^s$ is one-dimensional.
     \item $T_{\sigma}M=E^s$ where $E^s$ is three-dimensional.
         \item $T_{\sigma}M=E^u$ where $E^u$ is three-dimensional.
\end{itemize}
The last two cases correspond to singularities in trivial chain recurrence classes.
For the first two cases,  Corollary~\ref{c.lorenzlike}, shows that if the singularity is in a nontrivial chain recurrence class, then  it is Lorenz like. This gives that in the first two cases, the periodic index of $\sigma$ is one for all $\sigma \in C$. These singularities have a hyperbolic continuation on a neighborhood $U$ of $C$ and a neighborhoods $\cU$
of $X$. This implies, by  taking  $U$ smaller if necessary, that all singularities in $U$ have the same periodic index for any vector field in $\cU$.

Only three kinds of hyperbolic splittings exist for a periodic orbit $p$ in a three-dimensional manifold,
\begin{itemize}
    \item $T_{p}M=E^s\oplus<X(p)>\oplus E^u$ where $E^s$ and $E^u$ are one-dimensional.
     \item $T_{p}M=E^s\oplus<X(p)>$ where $E^s$ is two-dimensional.
      \item $T_{p}M=<X(p)>\oplus E^u$ where $E^u$ is two-dimensional.
\end{itemize}
In both last cases, the chain recurrence class of $p$ is itself.

Suppose now that all periodic orbits in C exhibit a splitting of the form $$T_{p}M=E^s\oplus<X(p)>\oplus E^u\,.$$ We are interested in chain recurrence classes C that exhibit singularities, since if not, \cite{W1} 
gives us that $X$ is uniformly hyperbolic in a neighborhood $U$ of $C$ and therefore homogeneous. 

Suppose for contradiction that there is a sequence of attracting periodic orbits  $\gamma_n$ of vector fields $Y_n\to X$ such that $\gamma_n\to_{H} S\subset C$ (if they were repellers, the proof is analogous). From  Lemma~\ref{GY},
we conclude that $S$ contains no singularities and, therefore, $S\subsetneq C$. Since $S$ is the Hausdorff limit of periodic orbits, $S$ is chain recurrent. 

Since $\gamma_n$  are attracting periodic hyperbolic orbits and by Proposition~\ref{Claim2} and Lemma I.5 in \cite{Ma2}), we have that for  $Y_n$ with $n>n_0$, there exists an ergodic measure $\nu_n$ supported in $\gamma_n$ such that
$$\int \log\norm{h^T_{-}.\psi^T\mid_{N_x}} d\nu_n(x)<0  \,.$$
The weak-*limit of $\nu_n$ is an invariant measure $\nu'$ for the flow of $X$. This implies that there exists an ergodic measure $\nu$ (in the ergodic decomposition of $\nu'$) that is supported in $S$ and for that measure
$$\int \log\norm{h^T_{-}.\psi^T\mid_{N_x}} d\nu(x)\leq 0 \,.$$
From Theorem~\ref{5.6}, we get the following: $$\int \log\norm{h^T_{-}.\psi^T\mid_{N_x}} d\nu(x)<0 \,,$$
This is impossible, since $S$ is strictly included in $C$.

\end{proof}

\begin{rema}[Remark 17 in \cite{BdL}]\label{singhip}
If a chain recurrence class $C$ of a multi-singular hyperbolic vector field $X$ has all singularities with the same index, it is singular hyperbolic. 
\end{rema}

Theorem~\ref{T.main1} 
and Remark~\ref{singhip} directly imply Theorem~\ref{T.main2}.

 Theorem~\ref{T.main3} is obtained as a consequence Theorem~\ref{T.main4}, along with the next two results found in \cite{GSW}. We restate them here. However, we replace the hypothesis of $X$ being a \emph{generic} star
vector field, with the generic condition used in the proof. This condition states that the chain recurrence classes of periodic points do not brake under perturbations. The proofs are identical and not short, so we will not repeat them here.

\begin{lemm}[\label{4.5}Lemma 4.5 and Theorem 5.7 in \cite{GSW}]
Let $X$ be a $C^1$ a star vector field and Let $C$ be a nontrivial chain recurrence class.
Suppose that $p$ and $q$ are two critical elements of $X$ in $C$ then there is a $\cC^1$ neighborhood $\cU$ of $X$ such that $C(p)=C(q)=C$.
Then there is a neighborhood $U$ of $C$ such that,
for every two  periodic points $p,q \in U$,

$$Ind(p) = Ind(q),$$
Furthermore, for any singularity $\sigma'$ in $U$,
$$Ind(\sigma') = Ind(q) \hspace{0.5cm}\text{ if } sv(\sigma')<0\,, $$
or
$$Ind(\sigma') = Ind(q)+1 \hspace{0.5cm}\text{ if } sv(\sigma)>0\,, $$
(i.e. $X$ is homogeneous in $U$). 
\end{lemm}

Since any robustly transitive chain recurrence class is such that if $p$ and $q$ are two critical elements of $X$ in $C$ then there is a $\cC^1$ neighborhood $\cU$ of $X$ such that  $C(p)=C(q)=C$, we get  Theorem~\ref{T.main3} as a direct consequence of Lemma~\ref{4.5} and Theorem~\ref{T.main4}.

 \section*{Acknowledgements}
 We would like to thank the Committee for Women in Mathematics (CWM) and  Fondecyt Iniciaci\'on project  Nº 11190815 (coordinated by Nelda Jaque Tamblay) for the financial support that allowed this work to occur. 
 The first author received support form JCNE/FAPERJ coordinated by Andres Koropecki for a research visit.
 
 We would like to thank the  hospitality of  IME-UFF, where this work was partially carried on.


\begin{thebibliography}{XXW}

\bibitem[A]{A} D. Anosov, \emph{Geodesic flows on closed Riemannian manifolds of negative curvature,}


\bibitem[AN]{AN} J.M. Alongi, G. S. Nelson, \emph{ Recurrence and topology,} GSM \textbf{85}, American Mathematical Society, 2007.


\bibitem[BaMo]{BaMo} S. Bautista, C.A. Morales, \emph{On the intersection of sectional-hyperbolic sets,} J. Mod. Dyn. \textbf{9} (2015), 203-218 . 


\bibitem[BC]{BC} C. Bonatti and S. Crovisier,  \emph{Récurrence et généricité,} Invent. Math. \textbf{158} (2004), no. 1, 33-104.



\bibitem[BdL]{BdL}C. Bonatti and A. da Luz, \emph{Star flows and multi-singular hyperbolicity,} J. Eur. Math. Soc. \textbf{23} (2021), no. 8, 2649–2705. 



\bibitem[BDP]{BDP}C. Bonatti, L. J. D\'iaz and E.  Pujals,  \emph{A $C^1$-generic dichotomy for diffeomorphisms: weak forms of hyperbolicity or infinitely many sinks or sources,} Ann.  of Math. \textbf{158} (2003), no. 2, 355-418. 


\bibitem[BDV]{BDV}C. Bonatti, L. J. D\'iaz and M.  Viana,  \emph{Dynamics beyond uniform hyperbolicity. A global geometric and probabilistic perspective,} Encyclopaedia of Mathematical Sciences \textbf{102}, Mathematical Physics, III. Springer-Verlag, Berlin (2005).


\bibitem[BGY]{BGY}C. Bonatti, S. Gan and D.  Yang,   \emph{Dominated chain recurrent class with singularities,} Ann. Sc. Norm. Super. Pisa Cl. Sci. \textbf{14} (2015), no. 1, 83-99. 


\bibitem[BV]{BV}C. Bonatti and M.  Viana, \emph{SRB measures for partially hyperbolic systems whose central direction is mostly contracting,} Israel J. Math. \textbf{115} (2000), 157-193. 


\bibitem[Co]{Co} C. Conley, \emph{Isolated invariant sets and the Morse index,} CBMS Regional Conference Series in Mathematics, vol. \textbf{38}, American Mathematical Society, Rhode Island, 1978.


\bibitem[C]{C}S. Crovisier,  \emph{Periodic orbits and chain-transitive sets of C1-diffeomorphisms,} Publ. Math. Inst. Hautes \'Etudes Sci., \textbf{104} (2006), 87-141.

\bibitem[CdLYZ]{CdLYZ}S. Crovisier, A. da Luz, D. Yang, J, Zhang, \emph{On the notion of singular domination and (multi-)singular hyperbolicity,} Publ. Sci. China Math., \textbf{63} (2020), 1721-1744.


\bibitem[dL]{dL}A. da Luz, \emph{Star flows with singularities of different indices,} arXiv:1806.09011.


\bibitem[D]{D} C. I. Doering, \emph{Persistently transitive vector fields on three-dimensional manifolds,} Dynamical systems and bifurcation theory (Rio de Janeiro, 1985), 59--89. Pitman Res. Notes Math. Ser. \textbf{160}, Longman Sci. Tech., Harlow (1987). 


\bibitem[GLW]{GLW}M. Li, S. Gan and L. Wen, \emph{Robustly transitive singular sets via approach of extended linear Poincar\'e flow,} Discrete Contin. Dyn. Syst. \textbf{13} (2005), 239-269. 


\bibitem[GSW]{GSW}Y. Shi,  S. Gan and L.  Wen,  \emph{On the singular-hyperbolicity of star flows,} J. Mod. Dyn. \textbf{8} (2014), no. 2, 191-219.


\bibitem[GW]{GW} J. Guckenheimer and R. Williams, \emph{Structural stability of Lorenz attractors,} Inst. Hautes Etudes Sci. Publ. Math. Ž , \textbf{50} (1979), 59-72.

\bibitem[GW1]{W1} S. Gan and L. Wen,\emph{ Nonsingualr star flows satiefy Axiom A and the no-cycle condition}, Invent. Math., 164 (2006), 279-315.
\bibitem[GWZ]{GWZ} S. Gan, L. Wen and Zhu, \emph{Indices of singularities of robustly transitive sets,} Discrete Contin. Dyn. Syst., \textbf{21} (2008), 945-957.


\bibitem[GY]{GY} S. Gan and D. Yang, \emph{Morse-Smale systems and horseshoes for three dimensional singular flows,} Ann. Sci. \'Ec. Norm. Sup\'er. \textbf{51} (2018), no. 1, 39-112.


\bibitem[H]{H}S. Hayashi,  \emph{Diffeomorphisms in $F^1(M)$ satisfy Axiom A,} Ergod. Th. Dynam. Sys., \textbf{12} (1992), 233-253.


\bibitem[H2]{H2} S. Hayashi, \emph{Connecting invariant manifolds and the solution of the $C^1$-stability and $\Omega$-stability conjectures for flows,} Ann. of Math. \textbf{145} (1997), no. 1, 81-137. 


\bibitem[K]{K} I. Kupka, \emph{Contribution à la théorie des champs génériques,} Contrib. Diff. Eqs., \textbf{2} (1963), 457-484

\bibitem[L]{L}S. Liao,  \emph{On $(\eta,d)$-contractible orbits of vector fields,} Systems Sci. Math. Sci., \textbf{2} (1989), 193-227.


\bibitem[L1]{liao-poincare} S. Liao, \emph{Certain ergodic properties of a differential system on a compact differentiable manifold,} Acta Sci. Natur. Univ. Pekinensis \textbf{9} (1963), 241-265. Front. Math. China \textbf{1} (2006), no. 1, 1-52. 


\bibitem[L2]{liao-star} S. Liao, \emph{A basic property of a certain class of differential systems,}  \emph{Acta Math. Sinica} \textbf{22} (1979), no. 3, 316--343.


\bibitem[L3]{liao-domination} S. Liao, \emph{On the stability conjecture,} Chinese Ann. Math. \textbf{1} (1980), 9-30.


\bibitem[Lo]{Lo} E. N. Lorenz,  \emph{Deterministic nonperiodic flow,} J. Atmosph. Sci., \textbf{20} (1963), 130-141.


\bibitem[M1]{M1} R. Ma\~n\'e,  \emph{An ergodic closing lemma,} Ann. of Math. \textbf{116} (1982), 503-540.


\bibitem[Ma2]{Ma2} R. Ma\~n\'e, \emph{A proof of the $C\sp 1$ stability conjecture,} Inst. Hautes \'etudes Sci. Publ. Math. \textbf{66} (1988), 161--210.


\bibitem[MPP]{MPP}C. Morales, M. Pacifico and E. Pujals, \emph{Robust transitive singular sets for $3$-flows are partially hyperbolic attractors or repellers,} Ann. of  Math. \textbf{160} (2004), 375-432.


\bibitem[PadM]{PadM} J. Palis, W. de Melo, (1982) \emph{The Kupka-Smale Theorem.} In: Geometric Theory of Dynamical Systems. Springer, New York, NY. https://doi.org/10.1007/978-1-4612-5703-5$\_$3. 


\bibitem[PaSm]{PaSm} J. Palis and S. Smale, \emph{ Structural stability theorems}, in 1970 Global Analysis  (Proc. Sympos. Pure Math., Vol. XIV, Berkeley, Calif., 1968), Amer. Math. Soc., Providence, R.I, 1970, 223--231.


\bibitem[P]{P} V. Pliss, \emph{The position of the separatrices of saddle-point periodic motions of systems of second order differential equations,} Differencialnye Uravnenija \textbf{7} (1971), 1199-1225.



\bibitem[R1]{R1} J. W. Robbin, \emph{A structural stability theorem,} Ann. of Math. (2) \textbf{94} (1971) 447-493.



\bibitem[R2]{R2} C. Robinson, \emph{Structural stability of $C^1$ diffeomorphisms,} J. Differential Equations \textbf{22} (1976), no. 1, 28--73.


\bibitem[S]{S} S. Smale, \emph{Differentiable dynamical systems,} Bull. Amer. Math. Soc. \textbf{73} (1967), 747-817.


\bibitem[S1]{S1} S. Smale,  \emph{Stable manifolds for differential equations and diffeomorphisms,} Ann. Scuola Normale Superiore Pisa, \textbf{18} (1963), 97-116.


\bibitem[W]{W}L. Wen, \emph{A uniform $C^1$ connecting lemma,} Discrete Contin. Dyn. Syst. \textbf{8} (2002), no. 1, 257-265. 

\bibitem[W2]{W2} L. Wen, \emph{On the $C^1$ stability conjecture for flows,} J. Differential Equatuins. \textbf{129} (1996), 334-357.


\bibitem[WX]{WX}L. Wen and Z. Xia, \emph{$C^1$ connecting lemmas,} Trans. Amer. Math. Soc. \textbf{352} (2000), no. 11, 5213-5230. 



\end{thebibliography}
\end{document}